\documentclass{article}
\usepackage{amsgen, amsmath, amsfonts, amsthm, amssymb, enumerate,amscd,color}
\usepackage[all]{xy}

\def\s#1#2{\langle \,#1 , #2 \,\rangle}

\definecolor{blue}{rgb}{0,0,1}
\definecolor{red}{rgb}{1,0,0}
\definecolor{green}{rgb}{0,1,0}

\newtheorem{defn}{Definition}[section]

\newtheorem{thm}[defn]{Theorem}

\newcommand {\C}{{\mathbb C}}

\newcommand {\G}{{\Gamma}}

\newcommand {\Z}{{\bf Z}}

\newcommand {\R}{{\mathbb R}}

\newcommand {\g}{{\gamma}}

\newcommand {\HH}{{\mathfrak  H}}

\newcommand {\ca}{{\mathbf a}}

\newcommand {\RE}{\text{Re}}

\def\PSL{\operatorname{PSL}}

\def\ca{{\mathfrak a}}
\def\cj{{\mathfrak a_j}}
\def\c1{{\mathfrak a_1}}
\def\cm{{\mathfrak a_m}}
\def\cb{{\mathfrak b}}

\def\ci{{\infty}}

\def\sa{{\sigma_\mathfrak a}}
\def\sj{{{\sigma_{\mathfrak a_j}}}}
\def\s1{{{\sigma_{\mathfrak a_1}}}}

\def\sb{{\sigma_\mathfrak b}}

\title{A new multiple Dirichlet series induced by a higher order form}
\author{
Anton Deitmar\& Nikolaos Diamantis\\ \ \\ 
Acta Arithmetica 142.4, 303-309 (2010)}
\date{}
\begin{document}

\maketitle

\section{Introduction}
This note proposes a new link between two relatively new objects in Number 
Theory, higher-order automorphic forms and multiple Dirichlet series. 

\it First-order automorphic forms \rm of weight $k \in 2{\mathbb Z_+}$ 
for a lattice $\G$ in $\PSL_2(\R)$
are defined as smooth complex-valued 
functions $f$ on the upper-half plane $\HH$ such that 
\begin{itemize}
\item $f|_k(\g-1)$ is a modular form of weight $k$ for $\G$  
\item $f|_k\pi=f$, for every parabolic element of $\G$ 
\item $f$ has a ``moderate growth at the cusps".
\end{itemize}

Here the action $|_k$ of $\PSL_2(\R)$ on 
functions $g:\HH \mapsto \C$ is defined by
$$(g|_k \gamma)(z)=g(\gamma z) (c z+d)^{-k}
$$
with $\gamma= \left ( \smallmatrix  * & * \\  c & d
 \endsmallmatrix \right )$ in $\PSL_2(\R)$.
We extend the action to the group ring $\mathbb C [PSL_2(\mathbb R)]$ by linearity.
This definition can be extended in a natural way to higher-order forms.
(Note that the order in this definition differs from that in the 
definitions given in previous papers in the subject. The reason for 
the modification of terminology is related to the fact that, by Fourier 
transform, $\mathbb Z$-supported tempered distributions are mapped to 
higher order $\mathbb Z$-invariants, where the natural differentiation order of the distributions corresponds to the new notion of order.)

Though some of the ideas behind the investigation of multiple Dirichlet 
series originated earlier, the systematic study
began in the mid-90's (\cite{BFH}, \cite{DGH} etc.) A 
definition of \it multiple Dirichlet series \rm is given in \cite{DGH}:
\begin{equation} 
\sum_{m_1=1}^{\infty}\dots \sum_{m_n=1}^{\infty}
\frac{1}{m_1^{s_1} \dots m_n^{s_n}} \int_0^{\infty} \dots \int_0^{\infty}
\frac{a(m_1, \dots, m_n, t_1, \dots, t_l)}{t_1^{w_1} \dots t_l^{w_l}} d 
t_1 \dots d t_l 
\end{equation} 
where $a(m_1, \dots, m_n, t_1, \dots, t_l)$ is a complex-valued smooth 
function, or, more generally, vectors with entries such series. Among 
these series, those that have a meromorphic 
continuation to the entire $\mathbb C^n$ 
and satisfy enough functional equations 
are of particular interest for applications and are sometimes refered to 
as `perfect'. 
Constructing `perfect' multiple 
Dirichlet series is much harder than the corresponding problem in 
classical Dirichlet series and it is one of the main aims of the theory.
Apart from the multiple Dirichlet series obtained from 
metaplectic Eisenstein series, essentially none of the known `perfect' 
multiple Dirichlet series are constructed as Mellin transforms but by 
other techniques, cf. \cite{BBCFH}, \cite{BBFH}, \cite{DG}, etc. In 
this note, we construct a perfect multiple Dirichlet series as the Mellin 
transform of a 
first-order form. This first-order form is essentially the Eisenstein 
series twisted by modular symbols. Although the resulting double 
Dirichlet series has infinitely many poles, and is thus not as suitable for 
current application as it would have been if it had finitely many poles, 
it is, to our knowledge, the first example of a non-classical modular 
object producing a double Dirichlet series via a Mellin transform. This 
may suggest that there may be a broader class of objects generating 
`perfect' multiple Dirichlet series in a systematic way similar to that 
between modular forms and Dirichlet series.

\bf Acknowledgment. \rm The authors thank Dorian Goldfeld for many 
useful conversations during work on this note.

\section{Eisenstein series twisted by modular symbols}
Let $\G \subset PSL_2(\R)$ be a non-uniform lattice.
As usual we write $x+iy=z\in
\HH$. 
Fix a set $\{\c1, \dots, \cm\}$ of
representatives of the inequivalent cusps of the group $\Gamma$.
For each $\cj$, we consider a scaling matrix $\sj$ such that
$\sj(\ci)=\cj$ and
$$
\sj^{-1} \G_\cj \sj= \G_\ci=
\left\{ \pm \left(\smallmatrix 1 & m \\ 0 & 1
\endsmallmatrix\right)
\; \big | \; \ m\in {\mathbb Z}\right\}
$$
where $\G_\cj$ is the stabilizer of $\cj$ in $\Gamma$.

Let $\psi:\Gamma\to\C$ be a group homomorphism which is zero on all parabolic elements.
For every $k \in 2\mathbb Z_+$ and $\ca\in\{\c1,\dots,\cm\}$, we set
$$
E_{\ca}(z, s, k; \psi):=\sum_{\g \in \G_{\ca} \backslash \G} \psi(\g) 
\text{Im}(\sa^{-1} \g z)^s j(\sigma_\ca^{-1}\gamma,z)^{-k},
$$
where $j(\gamma,z)=cz+d$ when $\gamma=\left(\begin{array}{cc}a & b \\c & 
d\end{array}\right)$.

This series is absolutely convergent for $\RE(s)>2-\frac k2$ and for 
$k=0$ it can be meromorphically continued
to all of $\C$ (\cite{GO}). It further satisfies
\begin{equation}
E_{\ca}(\g z, s, k; \psi)j(\g, z)^{-k}=E_{\ca}(z, s, k;
\psi)+\psi(\g^{-1})E_\ca(z,s,k)
\label{tl}
\end{equation}
where 
$$E_{\ca}(z, s, k)=\sum_{\g \in \G_{\ci} \backslash \G}
\text{Im}(\sigma_\ca^{-1}\g z)^sj(\sigma_\ca^{-1}\g, z)^{-k}$$
is the classical Eisenstein  series at $\ca$.

The function $E_{\ca}(z, s, k; \psi)$ is a weight $k$ 
first-order
automorphic form for each $s$ for which it is defined
(cf. \cite{CDO}, where first order is called second-order according 
to an older convention) and it has been used to obtain information
about the distribution of modular symbols (\cite{G}, \cite{PR}),
the number of appearances of a given generator in reduced words of 
$\Gamma$ etc.

Consider now the lattice
$$\G^*=\langle \G_0(N), W_N \rangle=\G_0(N) \cup \G_0(N)W_N$$
where $\G_0(N)=\{\left ( \smallmatrix  a & b \\  c & d
 \endsmallmatrix \right ) \in SL_2(\Z); N|c\}$ and
 $W_N=\left ( \smallmatrix  0 & -1/\sqrt N \\  \sqrt N & 0
 \endsmallmatrix \right )$. 

We now concentrate on the case $k=0$ for simplicity and because the weight 
does not affect the point we want to make. Let $f$ be a newform of 
weight $2$ for $\G_0(N)$ such that
 $L_f(1)=0$. We set $$\psi(\g)
 =\langle f, \g \rangle:=\int_{i \ci}^{\g i \ci}f(z)dz$$
and $E_{\ca}(z, s; f):=E_{\ca}(z, s; \psi)$.
The Fourier expansion of $E_{\ca}(z, s; f)$ at $\cb$ is
$$E_{\ca}(\sb z, s; f)=\phi_{\ca \cb}(s; f)y^{1-s}+\sum_{n \ne 0}
\phi_{\ca \cb}(n, s; f)W_s(nz)$$
with
$$
\phi_{\ca \cb}(s; f)=\pi \frac{\G(s-\frac{1}{2})}{\G(s)}
\sum_{c \in C_{\ca \cb}}\frac{S_{\ca \cb}(0, 0, f; c)}{c^{2s}}$$
and
$$\phi_{\ca \cb}(n, s; f)=\frac{\pi^s}{\G(s)}|n|^{s-1}
\sum_{c \in C_{\ca}}\frac{S_{\ca \cb}(n, 0, f; c)}{c^{2s}}
$$
where $C_{\ca \cb}=\{c>0; \left ( \smallmatrix  * & * \\  c & *
 \endsmallmatrix \right ) \in \sa^{-1} \G^* \sb \}$ and
 $$S_{\ca \cb}(m, n, f; c)=
 \sum_{\substack{\g \in \G_{\ci} \backslash \sa^{-1} \G^* \sb /\G_{\ci}
\\ \g_c=c}} 
<f, \sa \g \sb^{-1}> e^{2 \pi i (n \frac{\g_{a}}{c}+
 m\frac{\g_d}{c})}.$$
We denote the matrix $(\phi_{\ca \cb}(s; f))$ by $\Phi(s; f)$.

Set
$$\mathbf{E}(z, s; f):=(E_{\c1}(z, s; f), \dots, E_{\cm}(z, s; f))^T 
$$
and
$$\mathbf{E}(z, s):=(E_{\c1}(z, s), \dots, 
E_{\cm}(z, s))^T$$ where $^T$ indicates matrix transpose. 
In (\cite{O'S}), it is proved that $\mathbf{E}(z, s; f)$ can be 
meromorphically continued in $s$ and that, 
for all $s$ for which $\Phi(s, f)$, $\Phi(1-s)$ and $\Phi(s)$ 
are defined,
\begin{equation}
\Phi(s)\mathbf{E}(z, 1-s; f)=\mathbf{E}(z, s; f)-\Phi(s, f)
\Phi(1-s)\mathbf{E}(z, s).
\label{f.e.}
\end{equation}
$\Phi(s)=(\phi_{\ca \cb}(s))$ is the scattering matrix of the standard 
Eisenstein 
series and satisfies
\begin{equation}
\Phi(s) \Phi(1-s; f)=-\Phi(s, f) \Phi(1-s)
\label{c.c.}
\end{equation}

\section{$L$-functions of $\mathbf E$}
It is possible to use (\ref{c.c.}) to obtain a standard functional 
equation in $s$ 
for $\mathbf{E}(z, s; f)$ (i.e. not ``shifted" by a multiple of $\mathbf 
E(z, s)$). Set 
$$\tilde{\mathbf{E}}(z, s; f):=
\mathbf{E}(z, s; f)
-\frac{1}{2}\Phi(s, f) \Phi(1-s)\mathbf{E}(z, s).
$$
Then, using the functional equation 
$\Phi(s)\mathbf{E}(z, 1-s)=\mathbf{E}(z, s)$, we obtain
\begin{multline} \Phi(s) \tilde{\mathbf{E}}(z, 1-s; f)
=\Phi(s)\mathbf{E}(z, 1-s; f)          
-\frac{1}{2}\Phi(s)\Phi(1-s, f) \Phi(s)\mathbf{E}(z, 1-s) \\
=\mathbf{E}(z, s; f)-\Phi(s, f)\Phi(1-s) \mathbf{E}(z, s)          
-\frac{1}{2}\Phi(s)\Phi(1-s, f) \Phi(s)\mathbf{E}(z, 1-s) \\
=\mathbf{E}(z, s; f)-\Phi(s, f)\Phi(1-s) \mathbf{E}(z, s)          
+\frac{1}{2}\Phi(s, f)\Phi(1-s) \mathbf{E}(z, s)=
\tilde{\mathbf{E}}(z, s; f).
\label{f.e.w}
\end{multline}

$\tilde{\mathbf{E}}(z, s; f)$ has infinitely many poles at 
Re$(s)=\frac{1}{2}$ and, possibly, finitely many poles in $(1/2, 1)$
because $\mathbf{E}(z, s)$ has no poles on Re$(s)=\frac{1}{2}$.

In terms of $z$, (\ref{tl}) and $L_f(1)=0$ imply that 
\begin{equation*}
\mathbf{E}(W_N z, s; f)=
\mathbf{E}(z, s; f)+\langle f, W_N \rangle \mathbf{E}(z, s)=
\mathbf{E}(z, s; f)
\end{equation*}
and thus 
\begin{equation}
\tilde{\mathbf{E}}(W_N z, s; f)=\tilde{\mathbf{E}}(z, s; f).
\label{tl1}
\end{equation}

We next define the ``completed" $L$-function of $\tilde{E}_{\ca}$. For 
every $w$ for which $\tilde{E}_{\ca}(iy, w; f)$ is defined and for 
Re$(s)>$max$(1+w, 2-w)$, set
 $$\tilde{\Lambda}_{\ca}(s, w)=\int_0^{\ci} \left 
( \tilde{E}_{\ca}(iy, w; f)-a_{\ca}(w)y^w-b_{\ca}(w)y^{1-w}\right ) y^s 
\frac{dy}{y}$$
where $a_{\ca}(w)y^w+b_{\ca}(w)y^{1-w}$ is the constant term of 
$\tilde{E}_{\ca}(z, w; f)$. We also set  
$$\tilde{\mathbf{\Lambda}}:=(\tilde{\Lambda}_{\c1}, \dots, 
\tilde{\Lambda}_{\cm}).$$
We shall prove 
\begin{thm} The function $\tilde{\mathbf{\Lambda}}(s, w)$
is a (vector-valued) double Dirichlet series which can be meromorphically 
continued to $\mathbb C^2$. If $\tilde{\mathbf{E}}(z, w; f)$ does not have 
a pole at $w_0$, then $\tilde{\mathbf{\Lambda}}(s, w)$ has a pole at $s= 
\pm w_0,$ and $s=\pm(w_0-1)$. It further satisfies the functional 
equations
\begin{equation*}
N^s\tilde{\mathbf{\Lambda}}(s, w)=
\tilde{\mathbf{\Lambda}}(-s, w).
\end{equation*}
and
\begin{equation*}
\Phi(w)\tilde{\mathbf{\Lambda}}(s, 1-w)=\tilde{\mathbf{\Lambda}}(s, w)
\end{equation*}
\end{thm}
\begin{proof}
(\ref{tl1}) implies 
that \begin{multline*} \tilde{\Lambda}_{\ca}(s, w)=
\int_{1/\sqrt N}^{\ci} \left (\tilde{E}_{\ca}(iy, w; 
f)-a_{\ca}(w)y^w-b_{\cb}(w)y^{1-w}\right ) y^s
\frac{dy}{y}+\\
\int_{1/\sqrt N}^{\ci} \left (\tilde{E}_{\ca}(i/(Ny), w; f)-
a_{\ca}(w)(Ny)^{-w}-b_{\cb}(w)(Ny)^{w-1}\right ) (Ny)^{-s}\frac{dy}{y}
=\\
\int_{1/\sqrt N}^{\ci} \left (\tilde{E}_{\ca}(iy, w; 
f)-a_{\ca}(w)y^w-b_{\cb}(w)y^{1-w}\right ) (y^s+N^{-s} y^{-s})
\frac{dy}{y}+
\\
N^{-s}\int_{1/\sqrt{N}}^{\ci} \left (
a_{\ca}(w)y^{w}+b_{\ca}(w)y^{1-w}-
a_{\ca}(w)(Ny)^{-w}-b_{\ca}(w)(Ny)^{w-1} \right ) 
y^{-s} \frac{dy}{y}=\\
\\
\int_{1/\sqrt N}^{\ci} \left (\tilde{E}_{\ca}(iy, w; 
f)-a_{\ca}(w)y^w-b_{\cb}(w)y^{1-w}\right ) (y^s+N^{-s} y^{-s})
\frac{dy}{y}+\\
\frac{a_{\ca}(w) (\sqrt N)^{-s-w}}{s-w}+
\frac{b_{\ca}(w) (\sqrt N)^{w-1-s}}{s+w-1}-
\frac{a_{\ca}(w) (\sqrt N)^{-s-w}}{s+w}-
\frac{b_{\ca}(w)(\sqrt N)^{-s+w-1}}{s-w+1}
\end{multline*}
Since for every $w$ for which 
$\tilde{E}_{\ca}(iy, w; f)$ is defined, $E_{\ca}(iy, 
w; f)-$constant term$=O(e^{-\pi y})$, and the analogous fact holds for 
$E_{\ca}(z, w)$, the last integral is 
well-defined and gives a holomorphic function in $s$. 
This shows that $\tilde{\Lambda}_{\ca}$ can be meromorphically continued 
to $\C^2$ with poles at 
$s=\pm w$ and $w-1=\pm s$, and that
$N^s\tilde{\Lambda}_{\ca}(s, w)=\tilde{\Lambda}_{\ca}(-s, w).$

Further, since the vector of constant terms of the entries of 
$\tilde{\mathbf{E}}(z, s; f)$ satisfies the same functional equation in 
$s$ as $\tilde{\mathbf{E}}(z, s; f)$ (that is, (\ref{f.e.w})), we 
immediately deduce that 
$\tilde{\mathbf{\Lambda}}(s, w)$ satisfies
\begin{equation}
\Phi(w)\tilde{\mathbf{\Lambda}}(s, 1-w)=\tilde{\mathbf{\Lambda}}(s, w)
\label{f.e.w.l}
\end{equation}

Finally, by the formulas for the Fourier coefficients of 
$E_{\ca}(z, s; f)$, $E_{\ca}(z, s)$,
and for the functions $\phi_{\ca \cb}(s), \phi_{\ca \cb}(s, f)$
we observe that $\tilde{\Lambda}_{\ca}(s, w)$ is a 
double 
Dirichlet series according the definition in the introduction.

Then the functional equations just proved imply the result.
\end{proof}

\bf Remark. \rm Incidentally, the fact that $\tilde{\mathbf{E}}(z, w; 
f)$, and thus $\tilde{\mathbf{\Lambda}}(s, w)$, has infinitely many poles 
in $w$ (\cite{O'S0})
shows that $\tilde{\mathbf{\Lambda}}$ is a `genuine' double 
Dirichlet 
series and not a finite sum of products of (one-variable) $L$-functions of 
classical modular forms.

\small
\begin{tabular}{ll}
Anton Deitmar & Nikolaos Diamantis\\
Mathematisches Institut & School of Mathematical Sciences\\
Auf der Morgenstelle 10 & University of Nottingham\\
72076 T\"ubingen & University Park\\
Germany & Nottingham NG7 2RD\\
{\tt deitmar@uni-tuebingen.de}& United Kingdom\\
& {\tt nikolaos.diamantis@maths.nottingham.ac.uk}
\end{tabular}

\end{document}